\documentclass[a4paper,11pt,fleqn]{article}

\usepackage{amsmath,amssymb,amsthm}

\newcommand{\noprint}[1]{}

\newtheorem{theorem}{Theorem}
\newtheorem{lemma}{Lemma}

\newtheorem{proposition}{Proposition}
{\theoremstyle{definition} \newtheorem{definition}{Definition}
\newtheorem{example}{Example}
\newtheorem{remark}{Remark}

\begin{document}

\par\noindent {\LARGE\bf
Canonical forms for matrices of Saletan contractions\par}

{\vspace{4mm}\par\noindent
Dmytro R. POPOVYCH
\par\vspace{2mm}\par}

{\vspace{2mm}\par\it\noindent
Faculty of Mechanics and Mathematics, National Taras Shevchenko University of Kyiv,\\
building 7, 2, Academician Glushkov Av., 03127 Kyiv, Ukraine,
\par}

\vspace{2mm}\par\noindent E-mail: deviuss@gmail.com

{\vspace{6mm}\par\noindent\hspace*{5mm}\parbox{150mm}{\small
We show that each Saletan (linear) contraction can be realized, up to change of bases of the initial and the target Lie algebras,
by a matrix-function that is completely defined by a partition of the dimension of Fitting component of its value at the limit value of the contraction parameter.
The codimension of the Fitting component and this partition constitute the signature of the Saletan contraction.
We study Saletan contractions with Fitting component of maximal dimension and trivial one-part partition.
All contractions of such kind in dimension three are completely classified.
}\par\vspace{4mm}}

\section{Introduction}

Historically, the first extensively studied kind of contractions of Lie algebras,
after Segal introduced the general notion of contractions~\cite{Segal1951},
was the class of Saletan (linear) contractions.
Contractions of Lie algebras became known as a tool of theoretical physics after the famous papers by
In\"on\"u and Wigner~\cite{Inonu&Wigner1953,Inonu&Wigner1954} on an important specific subclass of linear contractions.
Note that In\"on\"u and Wigner planned to consider the whole class of linear contractions
but they erroneously claimed in~\cite{Inonu&Wigner1953} that any linear contraction is diagonalizable.
Even though In\"on\"u and Wigner corrected their considerations in the next paper~\cite{Inonu&Wigner1954},
they proceeded to exclusively study diagonalizable linear contractions,
which due to their contribution are now called \emph{In\"on\"u--Wigner contractions}.
The effectiveness of such contractions in applications is ensured by their close connection
to subalgebras of initial algebras.
More precisely, in modern terms the main result of~\cite{Inonu&Wigner1953},
which is Theorem 1 at ~\cite[p.~513]{Inonu&Wigner1953}, can be reformulated in the following way:
Any In\"on\"u--Wigner contraction of a Lie algebra~$\mathfrak g$ to a Lie algebra~$\mathfrak g_0$ is associated
with a subalgebra of~$\mathfrak g$, say~$\mathfrak s$, and starting with an arbitrary subalgebra of the algebra~$\mathfrak g$
one can construct an In\"on\"u--Wigner contraction of this algebra. In the contracted algebra~$\mathfrak g_0$
there exists an Abelian ideal~$\mathfrak i$ such that the factor-algebra $\mathfrak g_0 /\mathfrak i$
is isomorphic to~$\mathfrak s$.

A thorough study of linear contractions was conducted by Saletan in the course of preparation of his doctoral thesis
and was published in~\cite{Saletan1961}.
In particular, Saletan obtained a simplified form for matrices of linear contractions up to reparametrization and basis change,
derived a criterion for a linear matrix-function to be a contraction matrix,
and gave the expression for the Lie bracket of the contracted Lie algebra.
He also studied iterated linear contractions, related characteristics of the contraction matrix with
the subalgebraic structure of the initial algebra, and discussed linear contractions of representations of Lie algebras.

Further studies by other authors extended rather than deepened Saletan's results.
Thus, In\"on\"u--Wigner contractions of three- and four-dimensional Lie algebras were classified~\cite{Conatser1972,Huddleston1978}
due to the subalgebraic structure of these algebras being known~\cite{Patera&Winternitz1977}.
Following Saletan, contractions realized by matrix-functions of the generalized form $A\varepsilon+B\varepsilon^p$,
where~$A$ and~$B$ are constant matrices and~$\varepsilon$ is the contraction parameter,
were considered in a similar fashion~\cite{Kupczynski1969,Levy-Nahas1967,Mimura&Ikushima1977}.
Linear contractions of general algebraic structures were studied in~\cite{Carinena&Grabowski&Marmo2000}.

In contrast to the above studies, this paper is aimed to enhance the original results by Saletan.
We find the canonical form of Saletan contraction matrices, which creates the basis for introducing the notion
of Saletan contraction's signature, for developing an algorithm for computation of Saletan contractions,
and for posing new problems in this field.

The structure of the paper is the following:
Basic notions and results on contractions and, specifically, on Saletan contractions are presented in Section~\ref{SectionOnAuxiliaryResults}.
The main result of the paper, Theorem~\ref{TheoremOnCanonicalForms}, which deals with the canonical form of Saletan contraction matrices, is proved in Section~\ref{SectionOnCanonicalFormsOfSaletanContractionMatrices}.
After defining the notion of Saletan signature, we relate the signature of a Saletan contraction with the nested chain of subalgebras of the initial algebra that corresponds to this contraction.
In Section~\ref{SectionOnSaletanContractionsWithSign0n} we carry out a preliminary study of Saletan contractions
associated with chains of nested subalgebras of the maximal possible length, which coincides with the algebra dimension.
Then we exhaustively describe such contractions between three-dimensional Lie algebras over the complex (resp. real) field.
In the final section, we discuss obtained results and propose new problems for the further investigation.

\section{Basic notions and auxiliary results}
\label{SectionOnAuxiliaryResults}

Given a finite-dimensional vector space $V$ over the field~$\mathbb F=\mathbb R$ or~$\mathbb F=\mathbb C$,
by $\mathcal L_n=\mathcal L_n(\mathbb F)$ we denote the set of all possible Lie brackets on~$V$, where $n=\dim V<\infty$.
Each element~$\mu$ of~$\mathcal L_n$ corresponds to a Lie algebra with the underlying space~$V$, $\mathfrak g=(V,\mu)$.
Fixing a basis $\{e_1,\dots,e_n\}$ of~$V$ leads to a bijection between $\mathcal L_n$ and the set of structure constant tensors
\[
\mathcal C_n=\{(c_{ij}^k)\in\mathbb F^{n^3}\mid c_{ij}^k+c_{ji}^k=0,\,
c_{ij}^{i'\!}c_{i'\!k}^{k'\!}+c_{ki}^{i'\!}c_{i'\!j}^{k'\!}+c_{jk}^{i'\!}c_{i'\!i}^{k'\!}=0\}.
\]
The structure constant tensor $(c_{ij}^k)\in\mathcal C_n$ associated with a Lie bracket $\mu\in\mathcal L_n$
is given by the formula $\mu(e_i,e_j)=c_{ij}^ke_k$.
Here and in what follows, the indices $i$, $j$, $k$, $i'$, $j'$, $k'$, $p$ and~$q$ run from 1 to $n$
and the summation convention over repeated indices is assumed.
The right action of the group~${\rm GL}(V)$ on $\mathcal L_n$, which is conventional for the physical literature,
is defined~as
\[
(U\cdot\mu)(x,y)=U^{-1}\bigl(\mu(Ux,Uy)\bigr)\quad \forall U\in {\rm GL}(V),\forall \mu\in\mathcal L_n,\forall x,y\in V.
\]

\begin{definition}\label{DefOfContractions1}
Given a Lie bracket~$\mu\in\mathcal L_n$ and a continuous matrix function $U\colon (0,1]\to {\rm GL}(V)$,
we construct the parameterized family of Lie brackets~$\mu_\varepsilon=\mu\circ U_\varepsilon$, $\varepsilon \in (0,1]$.
Each Lie algebra $\mathfrak g_\varepsilon=(V,\mu_\varepsilon)$ is isomorphic to $\mathfrak g=(V,\mu)$.
If the limit
\[
\lim\limits_{\varepsilon \to +0}\mu_\varepsilon(x,y)=
\lim\limits_{\varepsilon \to +0}U_\varepsilon{}^{-1}\mu(U_\varepsilon x,U_\varepsilon y)=:\mu_0(x,y)
\]
exists for any $x, y\in V$, then $\mu_0$ is a well-defined Lie bracket.
The Lie algebra $\mathfrak g_0=(V,\mu_0)$ is called a \emph{one-parametric continuous contraction}
(or simply a \emph{contraction}) of the Lie algebra~$\mathfrak g$.
We call a limiting process that provides $\mathfrak g_0$ from~$\mathfrak g$ with a matrix function
a \emph{realization} of the contraction $\mathfrak g\to\mathfrak g_0$.
\end{definition}

The notion of contraction is extended to the case of an arbitrary algebraically closed field
in terms of orbit closures in the variety of Lie brackets, see, e.g.,
\cite{Burde1999,Burde2005,Burde&Steinhoff1999,Grunewald&Halloran1988,Lauret2003}.

If a basis~$\{e_1, \ldots, e_n\}$ of~$V$ is fixed,
then the operator $U_\varepsilon$ can be identified with its matrix $U_\varepsilon\in{\rm GL}_n(\mathbb F)$,
which is denoted by the same symbol,
and Definition~\ref{DefOfContractions1} can be reformulated in terms of structure constants.
Let $C=(c^k_{ij})$ be the tensor of structure constants of the algebra~$\mathfrak g$ in the basis chosen.
Then the tensor $\smash{C_\varepsilon=({c}^k_{\varepsilon,ij})}$ of structure constants of the algebra~$\mathfrak g_\varepsilon$ in this basis
is the result of the action by the matrix~$U_\varepsilon$ on the tensor~$C$, $C_\varepsilon=C\circ U_\varepsilon$.
In term of components this means that
\[c^k_{\varepsilon,ij}=(U_\varepsilon)^{i'}_i(U_\varepsilon)^{j'}_j(U_\varepsilon{}^{-1})^k_{k'}c^{k'}_{i'\!j'}.\]
Then Definition~\ref{DefOfContractions1} is equivalent to that the limit
\[\lim\limits_{\varepsilon\to+0}c^k_{\varepsilon,ij}=:c^k_{0,ij}\]
exists for all values of $i$, $j$ and $k$ and, therefore,
$c^k_{0,ij}$ are components of the well-defined structure constant tensor~$C_0$ of the Lie algebra~$\mathfrak g_0$.
The parameter $\varepsilon$ and the matrix-function $U_\varepsilon$ are called a \emph{contraction parameter} and a \emph{contraction matrix},
respectively.

The following useful assertion is obvious.

\begin{lemma}\label{LemmaOnFactorizationOfContractionMatrix}
If the matrix~$U_\varepsilon$ of a contraction $\mathfrak g\to\mathfrak g_0$ is
represented in the form $U_\varepsilon=\hat U_\varepsilon\check U_\varepsilon$, where
$\hat U$ and $\check U$ are continuous functions from $(0,1]$ to ${\rm GL}_n(\mathbb F)$ and
the function~$\check U$ has a limit~$\check U_0\in{\rm GL}_n(\mathbb F)$ at $\varepsilon \to +0$,
then $\hat U_\varepsilon\check U_0$ also is a matrix of the contraction $\mathfrak g\to\mathfrak g_0$.
\end{lemma}

\begin{remark}\label{RemarkOnFactorizationOfContractionMatrix}
Lemma~\ref{LemmaOnFactorizationOfContractionMatrix} implies that
$\hat U_\varepsilon$ is a matrix of the equivalent contraction $\mathfrak g\to\tilde{\mathfrak g}_0$,
where $\tilde{\mathfrak g}_0=(V,\mu_0\circ\check U_0^{-1})$ is the algebra isomorphic to $\mathfrak g_0$ with respect to the matrix $U_0^{-1}$.
\end{remark}

Historically, the first contractions studied were the ones realized by linear matrix-functions.

\begin{definition}
A realization of a contraction with a matrix-function that is linear in the contraction parameter
is called a \emph{Saletan (linear) contraction}~\cite{Saletan1961}.
\end{definition}

This class of contractions includes the In\"on\"u--Wigner contractions~\cite{Inonu&Wigner1953,Inonu&Wigner1954,Saletan1961}.

The matrix of any linear contraction has a well-defined limit at $\varepsilon=0$.
This is why in contrast to the general definition of contractions, in the case of a linear contraction its matrix-function~$U_\varepsilon$
can be assumed to be defined on the closed interval $[0,1]$.
Then it is convenient to represent the matrix $U_\varepsilon$ in the form $U_\varepsilon=(1-\varepsilon)U_0+\varepsilon U_1$,
where~$U_0$ and~$U_1$ are the values of~$U_\varepsilon$ at $\varepsilon=0$ and $\varepsilon=1$, respectively~\cite{Saletan1961}.
By the definition of contraction matrix, the matrix~$U_1$ is nonsingular,
and, for proper contractions, the matrix $U_0$ is necessarily singular.

There exist specific reparametrizations that preserve the class of Saletan contractions~\cite{Saletan1961}.
Let $U_\varepsilon=B+\varepsilon A$ be the matrix of a Saletan contraction.
We fix $\lambda>-1$ and consider the matrix-function~$U_\varepsilon$ on the interval $[0,(1+\lambda)^{-1}]$
instead of $[0,1]$. Then
\[
B+\varepsilon A=(1-\lambda\varepsilon)B+\varepsilon(A+\lambda B)=(1-\lambda\varepsilon)\left(B+\frac\varepsilon{1-\lambda\varepsilon}(A+\lambda B)\right).
\]
The multiplier $(1-\lambda\varepsilon)$ is not essential since its limit at $\varepsilon=0$ equals $1$.
Removing this multiplier and denoting $\varepsilon/(1-\lambda\varepsilon)$ by $\tilde\varepsilon$,
we obtain the well-defined linear matrix-function
\[
\tilde U_{\tilde\varepsilon}=B+\tilde\varepsilon(A+\lambda B), \quad \tilde\varepsilon\in [0,1],
\]
 which realizes the same Saletan contraction as~$U_\varepsilon$.

\section{Canonical forms of Saletan contraction matrices}
\label{SectionOnCanonicalFormsOfSaletanContractionMatrices}

We denote the $m\times m$ unit matrix by $E^m$, and $m\times m$ Jordan block with an eigenvalue $\lambda$ by $J_\lambda^m$.

\begin{theorem}\label{TheoremOnCanonicalForms}
Up replacing the algebras $\mathfrak g$ and $\mathfrak g_0$ with isomorphic ones,
every Saletan contraction $\mathfrak g\to\mathfrak g_0$ is realized by a matrix of the canonical form
\begin{equation}\label{EqCanonicalFormsOfLinearContraction}
E^{n_0}\oplus J_\varepsilon^{n_1}\oplus\dots\oplus J_\varepsilon^{n_s},
\qquad\mbox{or, equivalently,}\qquad
E^{n_0}\oplus J_0^{n_1}\oplus\dots\oplus J_0^{n_s}+\varepsilon E^n,
\end{equation}
where $n_0+\dots+n_s=n$.
\end{theorem}

\begin{proof}
The initial steps of the proof follow~\cite{Saletan1961}.
As the contraction matrix~$U_\varepsilon$ is linear in $\varepsilon$,
it admits the representation $U_\varepsilon=(1-\varepsilon)U_0+\varepsilon U_1$, where~$U_0$ and~$U_1$ are the values of~$U_\varepsilon$
at $\varepsilon=0$ and $\varepsilon=1$, respectively.
Taking the algebra~$\mathfrak g_1=(V,\mu_1)$ instead of~$\mathfrak g$ as the initial algebra of the contraction
leads to the new matrix $U_1=E^n$.
We restrict the range of the parameter~$\varepsilon$ to $[0, 1/2]$ and factor out the multiplier $1-\varepsilon$,
which can be canceled due to Lemma~\ref{LemmaOnFactorizationOfContractionMatrix} as it has limit~$1$ at $\varepsilon=0$.
We reparameterize the contraction matrix by introducing the new parameter $\tilde\varepsilon=\varepsilon/(1-\varepsilon)$,
whose range is $[0,1]$, and we omit the tilde over $\varepsilon$.
As a result, the contraction matrix takes the form $U_0+\varepsilon E^n$, which was derived in~\cite{Saletan1961}.

We carry out the Fitting decomposition of the space~$V$ relative to the operator~$U_0$.
More specifically, we partition the space $V$ into the direct sum of subspaces~$V_0$ and~$V_1$, $V=V_1\oplus V_0$
so that the restriction~$W_0$ of the operator~$U_0$ on the Fitting null component~$V_0$ is nilpotent
and the restriction~$W_1$ of this operator on the Fitting one component~$V_1$ is nonsingular.
The partition of the space~$V$ induces the partition of the new operator $U_\varepsilon$,
\[
U_\varepsilon=U_0+\varepsilon E^n=(W_1+\varepsilon E^{n_0})\oplus(W_0+\varepsilon E^{n-n_0}),
\]
where $n_0=\dim V_1$ is the dimension of the Fitting one component relative to~$U_0$.

Consider the matrices $\hat U_\varepsilon=E^{n_0}\oplus(W_0+\varepsilon E^{n-n_0})$ and
$\check U_\varepsilon=(W_1+\varepsilon E^{n_0})\oplus E^{n-n_0}$.
The representation $U_\varepsilon=\hat U_\varepsilon\check U_\varepsilon$ satisfies the conditions of Lemma~\ref{LemmaOnFactorizationOfContractionMatrix}.
Hence in view of Remark~\ref{RemarkOnFactorizationOfContractionMatrix}
$\hat U_\varepsilon$ is a matrix of the equivalent contraction $\mathfrak g_1\to\tilde{\mathfrak g}_0$,
where $\tilde{\mathfrak g}_0=(V,\mu_0\circ\check U_0^{-1})$ is the algebra isomorphic to $\mathfrak g_0$ with respect to the matrix $\check U_0^{-1}$.

Replacing the basis elements spanning the subspace~$V_0$, we can reduce the nilpotent matrix~$W_0$ to its Jordan form,
which is $J_0^{n_1}\oplus\dots\oplus J_0^{n_s}$ for some $n_1$, \dots, $n_s$ with $n_1+\dots+n_s=\dim V_0=n-n_0$.
Then the matrix $W_0+\varepsilon E^{n-n_0}$ takes its Jordan form $J_\varepsilon^{n_1}\oplus\dots\oplus J_\varepsilon^{n_s}$,
which is equivalent to representing the contraction matrix~$\hat U_\varepsilon$ in the first canonical form.

In view of Lemma~\ref{LemmaOnFactorizationOfContractionMatrix}, instead of the matrix~$\hat U_\varepsilon$ we can consider the matrix
$E^{n_0}\oplus W_0+\varepsilon E^n$, which differs from~$\hat U_\varepsilon$ by the right multiplier $(1+\varepsilon)E^{n_0}\oplus E^{n-n_0}$
with the unit matrix~$E^n$ as its limit at $\varepsilon=0$.
Then the above basis change in~$V_0$ results in the second canonical form of the linear contraction matrix~$U_\varepsilon$.
\end{proof}

\begin{definition}
Theorem~\ref{TheoremOnCanonicalForms} means that any Saletan contraction can be realized by a matrix of the form
$AS_\varepsilon B$, where~$A$ and~$B$ are constant nonsingular matrices
and the linear matrix-function~$S_\varepsilon$ is in the canonical form~\eqref{EqCanonicalFormsOfLinearContraction}.
Then the tuple $(n_0;n_1,\dots,n_s)$,
where $n_1$, \dots, $n_s$ constitute a partition of the dimension~$n-n_0$ of the Fitting null component relative to~$U_0$
and $n_0\in\{0,\dots,n\}$, is called the \emph{signature} of this Saletan contraction.
\end{definition}

Due to containing a partition of~$n-n_0$, the signature of a Saletan contraction is defined up to permutation of its parts excluding the first one.
Saletan contraction with signature $(n)$ is improper, i.e., the contracted algebra is isomorphic to the initial one.
So, for a proper Saletan contraction we necessarily have $n_0<n$.
In\"on\"u--Wigner contractions are associated with Saletan signatures of the form $(n_0; 1,\dots,1)$.
The Saletan signature $(0; 1,\dots,1)$ corresponds to the trivial contraction to the Abelian algebra.

The necessary and sufficient condition for the algebra~$\mathfrak g$ to be contracted by the linear matrix-function~$U_\varepsilon$~\cite{Saletan1961} is
\begin{gather}\label{EqSaletanCondition}
U_0^2[x,y]^0-U_0[U_0x,y]^0-U_0[x,U_0y]^0+[U_0x,U_0y]^0=0 \qquad \forall x,y\in V.
\end{gather}
Here and in what follows $[\cdot,\cdot]^0$ and $[\cdot,\cdot]^1$ denote the projections of the Lie brackets $[\cdot,\cdot]$
on the subspaces $V_0$ and $V_1$, respectively, which are not, in general, Lie brackets.
Then the contracted Lie bracket is defined by
\[
[x,y]_0=W_1^{-1}[U_0x,U_0y]^1-W_0[x,y]^0+[U_0x,y]^0+[x,U_0y]^0 \qquad \forall x,y\in V.
\]

The use of the canonical form of~$U_\varepsilon$ simplifies analysis of both the necessary and sufficient conditions and
properties of the contracted Lie bracket.
In particular, then $W_1^{-1}=E^{n-n_0}$.
We would like to emphasize that changing the basis of the underlying space without applying Lemma~\ref{LemmaOnFactorizationOfContractionMatrix}
can simplify the matrix~$W_1$ only to its Jordan form.

\begin{remark}\label{RemarkOnConnectionBtwnSignatureAndAlgStructure}
If $U_0$ is the value of the matrix of a well-defined Saletan contraction of the Lie algebra~$\mathfrak g$ at $\varepsilon=0$,
then each power of~$U_0$ is the value of the matrix of another well-defined Saletan contraction~$\mathfrak g$ at $\varepsilon=0$.
The image $\mathop{\rm im} U_0$ of~$U_0$ is a subalgebra of $\mathfrak g$.
Combining the above two claims, we have that
for each $m=0,1,2,\dots$ the image $\mathfrak s_m:=\mathop{\rm im} U_0^m$ of $U_0^m$ is also a subalgebra of $\mathfrak g$,
and $\mathfrak s_m=V_1$ if $m\geqslant m_0:=\max(n_1,\dots,n_s)~\cite{Saletan1961}$.
In other words, the matrix of any Saletan contraction is associated with the chain of nested subalgebras
\[
\mathfrak s_0:=\mathfrak g\supset\mathfrak s_1\supset\mathfrak s_2\supset\dots\supset\mathfrak s_{m_0}=V_1.
\]
The dimensions of these subalgebras are completely defined by the contraction signature,
\[
\dim\mathfrak s_m=n-l_1-\dots-l_m, \quad m=0,\dots,m_0,\quad\mbox{where}\quad l_m:=|\{n_i\mid n_i\geqslant m, \,i=1,\dots,s\}|.
\]
In particular, $\dim\mathfrak s_{m_0}=n_0$.
The above relation establishes necessary conditions of consistency between the structure of a Lie algebra and signatures of its Saletan contractions.
\end{remark}

\begin{example}\label{ExampleOnSaletanContractionsOfSO3}
Consider the real three-dimensional orthogonal algebra ${\rm so}(3)$ with the canonical commutation relations
$[e_1,e_2]=e_3$, $[e_2,e_3]=e_1$, $[e_3,e_1]=e_2$.
The algebra ${\rm so}(3)$ has no two-dimensional subalgebras.
Therefore, the only possible signature of a proper Saletan contraction of ${\rm so}(3)$ is $(1,1,1)$.
The first canonical form of the contraction matrix with this signature is $E^1\oplus J^1_\varepsilon\oplus J^1_\varepsilon$,
which realizes the single In\"on\"u--Wigner contraction of ${\rm so}(3)$,
which is to the Euclidean algebra ${\rm e}(2)$ defined by the nonzero commutation relations $[e_1,e_3]=-e_2$, $[e_2,e_3]=e_1$.
This implies that the other proper contraction of ${\rm so}(3)$,
which is to the Heisenberg algebra ${\rm h}(1)$ with the nonzero commutation relations $[e_2,e_3]=e_1$, cannot be realized as a Saletan contraction,
cf.~\cite{Saletan1961}.
\end{example}

\section{Saletan contractions with signature $\boldsymbol{(0;n)}$}
\label{SectionOnSaletanContractionsWithSign0n}

There are two different ways of studying Saletan contractions.
Given a fixed pair of Lie algebras, one can check whether there exists a Saletan contraction between these algebras
and then try to describe all possible Saletan contractions for this pair.
The other way is to describe all possible contractions which are realized by contraction matrices with certain signature.
A disadvantage of this way is the necessity of classifying Lie algebras that satisfy specific constraints.

In this section we consider contractions with the signature $(0;n)$.
Choosing this signature poses the most restrictive constraints on the structure of the initial Lie algebra~$\mathfrak g$
compared to other Saletan signatures, cf.\ Eq.~\eqref{EqSaletanCondition}.
In particular, the algebra~$\mathfrak g$ should contain a nested chain of~$n$ nonzero Lie subalgebras,
and, in general, this condition on~$\mathfrak g$ is not sufficient.%
\footnote{%
This strongly differs from the case of In\"on\"u--Wigner contractions, for which
there exists a one-to-one correspondence with proper subalgebras of~$\mathfrak g$.
}

For the signature $(0;n)$ we have $V_0=V$ and we can set
\[U_0=J_0^n.\]
Then $[\cdot,\cdot]^0=[\cdot,\cdot]$ and the Saletan condition~\eqref{EqSaletanCondition} takes the form
\[
[U_0x,U_0y]-U_0[U_0x,y]=U_0([x,U_0y]-U_0[x,y]) \qquad \forall x,y\in V.
\]
or, equivalently, $[\![{\rm ad}_{U_0x},U_0]\!]=U_0[\![{\rm ad}_x,U_0]\!]$ for any~$x\in V$.
Here and in what follows, $[\![A,B]\!]$ denotes the commutator of operators~$A$ and~$B$, $[\![A,B]\!]:=AB-BA$.
Specifying this condition for basis elements, for which $U_0=J_0^n$, we derive
$
[\![{\rm ad}_{e_i}, U_0]\!]=U_0[\![{\rm ad}_{e_{i+1}}, U_0]\!]=\dots=U_0^{n-i}[\![{\rm ad}_{e_n}, U_0]\!],
$
i.e.,
\begin{gather}\label{EqSaletanConditionForBasisElements}
[\![{\rm ad}_{e_i}, U_0]\!]=U_0^{n-i}[\![A, U_0]\!],
\end{gather}
where we use the notation $A=(a_{ij}):={\rm ad}_{e_n}$.
For each fixed~$i$, the equation~\eqref{EqSaletanConditionForBasisElements} can be considered
as inhomogeneous linear system of algebraic equations with respect to entries of the matrix ${\rm ad}_{e_i}$.
A particular solution of this system is given by $U_0^{n-i}A$, since $[\![U_0^kA, U_0]\!]=U_0^kAU_0-U_0^{k+1}A=U_0^k[\![A, U_0]\!] $.
The solution space of the corresponding homogeneous system $[\![{\rm ad}_{e_i}, U_0]\!]=0$ coincides with the space
of matrices commuting with~$U_0$, which is spanned by the powers of~$U_0$ due to $U_0$ being a single Jordan block.
Therefore, the general solution of the system~\eqref{EqSaletanConditionForBasisElements} is
\[
{\rm ad}_{e_i}=U_0^{n-i}A+b_{ij}U_0^{n-j}
\]
with parameters~$b_{ij}$, where $b_{nj}=0$ as ${\rm ad}_{e_n}=A$ by definition and $a_{in}=0$ as $Ae_n=[e_n,e_n]=0$.
Recall that we assume the summation convention over repeated indices.
The Lie bracket is skew-symmetric, which implies
\[
[e_i,e_n]+[e_n,e_i]={\rm ad}_{e_i}e_n+{\rm ad}_{e_n}e_i=b_{ij}e_j+a_{ji}e_j=0,
\]
i.e., $b_{ij}+a_{ji}=0$.
In other words, the commutation relations of the algebra~$\mathfrak g$ are
\begin{gather}\label{EqCommRelationsForSignature0n}
[e_i,e_j]={\rm ad}_{e_i}e_j=U_0^{n-i}Ae_j-a_{ki}U_0^{n-k}e_j=a_{kj}e_{k+i-n}-a_{ki}e_{k+j-n}
\nonumber\\
\phantom{[e_i,e_j]}=(a_{p+n-i,j}-a_{p+n-j,i})e_p.
\end{gather}
Here and in what follows, if an index goes beyond the index interval~$\{1,\dots,n\}$, then the corresponding object is assumed zero.
Thus, in view of~\eqref{EqCommRelationsForSignature0n}
the skew-symmetric property of the Lie bracket obviously holds for any pair of elements of~$\mathfrak g$.
Note that the number of essential parameters in the above commutation relations does not exceed $n(n-1)$.
The Jacobi identity imposes more constraints in the form of a system of quadratic equations with respect to entries of the matrix~$A$,
\begin{gather*}
(a_{p+n-i,j}-a_{p+n-j,i})(a_{q+n-k,p}-a_{q+n-p,k})+{}\\
(a_{p+n-j,k}-a_{p+n-k,j})(a_{q+n-i,p}-a_{q+n-p,i})+{}\\
(a_{p+n-k,i}-a_{p+n-i,k})(a_{q+n-j,p}-a_{q+n-p,j})=0.
\end{gather*}
Unfortunately, we were not able to solve this system for an arbitrary dimension of the underlying space.

The contracted Lie bracket is defined by
$[x,y]_0=[U_0x,y]+[x,U_0y]-U_0[x,y]$ for all $x,y\in V$.
Hence, the commutation relations of the contracted algebra~$\mathfrak g_0$ are
\begin{gather*}
[e_i,e_j]_0=[e_{i-1},e_j]+[e_i,e_{j-1}]-U_0[e_i,e_j]
\\\phantom{[e_i,e_j]_0}
=(a_{p+n-i+1,j}-a_{p+n-j,i-1})e_p+(a_{p+n-i,j-1}-a_{p+n-j+1,i})e_p
\\\phantom{[e_i,e_j]_0={}}
-(a_{p+n-i,j}-a_{p+n-j,i})e_{p-1}
\\\phantom{[e_i,e_j]_0}
=(a_{p+n-i,j-1}-a_{p+n-j,i-1})e_p,
\end{gather*}
In particular, $[e_n,e_j]_0=(a_{p,j-1}-a_{p+n-j,n-1})e_p$.
Consider the matrix $A_0=(a_{0,ij})$, where $a_{0,ij}=a_{i,j-1}-a_{i+n-j,n-1}$.
In terms of~$A$ and~$J_0^n$ we have the representation
\[
A_0=AJ_0^n-\sum_{i=0}^{n-1}a_{n-i,n-1}(J_0^n)^i.
\]
Roughly speaking, the matrix~$A_0$ is obtained from the matrix~$A$ by shifting the columns of~$A$ to the right,
filling of the first column by zeros and subtracting a specific linear combination of powers of $J_0^n$
that gives zeros in the last column of~$A_0$.
The structure of the algebra~$\mathfrak g_0$ is defined in terms of the matrix~$A_0$ in the same way as
the structure of the algebra~$\mathfrak g$ is defined in terms of the matrix~$A$
since $a_{p+n-i,j-1}-a_{p+n-j,i-1}=a_{0,p+n-i,j}-a_{0,p+n-j,i}$.
This is consistent with Lemma~3 of~\cite{Saletan1961}.
Indeed, as the algebra~$\mathfrak g_0$ can be contracted by the same matrix $U=J_\varepsilon^n$,
its structure constants satisfy the same constraints imposed by the Saletan conditions~\eqref{EqSaletanCondition}.
Lemma~3 of~\cite{Saletan1961} also implies that~$n$ iterations of this contraction leads to the Abelian algebra.

We exhaustively study the case~$n=3$.
There are three essential relations among the commutation relations~\eqref{EqCommRelationsForSignature0n} with $n=3$,
\begin{gather*}
[e_3,e_1]=a_{p1}e_p,\\
[e_3,e_2]=a_{p2}e_p,\\
[e_1,e_2]=(a_{32}-a_{21})e_1-a_{31}e_2,
\end{gather*}
and the single Jacobi identity
$[e_1,[e_2,e_3]]+[e_2,[e_3,e_1]]+[e_3,[e_1,e_2]]=0$.
Collecting coefficients of basis elements in the Jacobi identity and making additional arrangements,
we obtain the following system of equations on entries of the matrix~$A$:
\begin{gather}\label{EqSystemOnEntriesOfA}
\begin{split}
&a_{31}a_{21}=0, \quad a_{31}a_{12}=0, \quad a_{31}(a_{11}-a_{22})=0, \quad a_{21}(2a_{32}-a_{21})=0, \\ &a_{32}(a_{11}-a_{22})+a_{12}a_{22}=0.
\end{split}
\end{gather}
A consequence of the system is $a_{21}(a_{11}+a_{22})=0$.

In order to simplify the form of the matrix~$A$, we can use the transition to a Lie algebra isomorphic to~$\mathfrak g$ or, equivalently,
changing the basis of the underlying space.
In view of problem's statement, admitted basis changes are those whose matrices commute with the matrix~$U_0=J_0^n$.
Therefore, each of such matrices is a linear combination of powers of~$U_0$,
\[S=\gamma(E+\alpha U_0+\beta U_0^2),\]
where $\alpha$, $\beta$ and $\gamma$ are arbitrary constants with $\gamma\ne0$ and~$E$ is the $3\times 3$ identity matrix.
The inverse of~$S$ is \[S^{-1}=\gamma^{-1}(E-\alpha U_0+(\alpha^2-\beta)U_0^2).\]
The expressions for entries of the transformed matrix~$\tilde A$ follow from those for the transformed Lie brackets $[e_3,e_1]^\sim$ and $[e_3,e_2]^\sim$.\
We have
\begin{gather*}
[e_3,e_1]^\sim=S^{-1}[Se_3,Se_1]=\gamma(a_{11}-\alpha a_{32}-\beta a_{31})e_1+\gamma a_{21}e_2+\gamma a_{31}e_3,\\
[e_3,e_2]^\sim=S^{-1}[Se_3,Se_2]=\gamma(a_{12}+\alpha(a_{11}-a_{22})-\beta a_{21})e_1\\
\phantom{[e_3,e_2]^\sim=}{}+\gamma(a_{22}+\alpha(a_{21}-a_{32})-\beta a_{31})e_2+\gamma(a_{32}+\alpha a_{31})e_3,
\end{gather*}
i.e.,
\[\arraycolsep=0ex
\begin{array}{ll}
\tilde a_{11}=\gamma(a_{11}-\alpha a_{32}-\beta a_{31}),\quad &\tilde a_{12}=\gamma(a_{12}+\alpha(a_{11}-a_{22})-\beta a_{21}),\\[1ex]
\tilde a_{21}=\gamma a_{21},                            \quad &\tilde a_{22}=\gamma(a_{22}+\alpha(a_{21}-a_{32})-\beta a_{31}),\\[1ex]
\tilde a_{31}=\gamma a_{31},                            \quad &\tilde a_{32}=\gamma(a_{32}+\alpha a_{31}),
\end{array}
\]

The contracted algebra $\mathfrak g_0$ is defined by the commutation relations
\begin{gather*}
[e_3,e_1]_0=-a_{32}e_1,\\
[e_3,e_2]_0=(a_{11}-a_{22})e_1+(a_{21}-a_{32})e_2+a_{31}e_3,\\
[e_1,e_2]_0=a_{31}e_1,
\end{gather*}

We study possible cases of the solutions of the system~\eqref{EqSystemOnEntriesOfA} up to allowed basis changes.

If $a_{31}\neq 0$, then the system~\eqref{EqSystemOnEntriesOfA} implies that $a_{21}=a_{12}=0$ and $a_{11}=a_{22}$.
Selecting certain values of the parameters~$\alpha$, $\beta$, and $\gamma$ of the basis transformation,
we can set $a_{32}=0$, $a_{11}=a_{22}=0$ and $a_{31}=-1$.
In other words, the commutation relations of the algebra~$\mathfrak g$ take the form
$[e_1,e_2]=e_2$, $[e_1,e_3]=e_3$ and $[e_2,e_3]=0$.
Hence the basis elements~$e_2$ and~$e_3$ span the maximal Abelian ideal of the algebra~$\mathfrak g$,
and the element~$e_1$ acts on this ideal as the identity operator, i.e. the algebra~$\mathfrak g$
is the almost Abelian algebra associated with the identity operator, which is denoted by $\mathfrak g_{3.3}$ in
Mubarakzyanov's classification of three-dimensional Lie algebras~\cite{Mubarakzyanov1963a}.
In contrast to Example~\ref{ExampleOnSaletanContractionsOfSO3}, in what follows we mostly use Mubarakzyanov's notations.%
\footnote{%
The classical Lie algebras ${\rm h}(1)$, ${\rm e}(2)$, ${\rm sl}(2,\mathbb R)$ and ${\rm so}(3)$
are denoted by Mubarakzyanov as $\mathfrak g_{3.1}$, $\mathfrak g_{3.5}^0$, $\mathfrak g_{3.6}$ and $\mathfrak g_{3.7}$, respectively.
}
For the contracted algebra $\mathfrak g_0$, the commutation relations are:
$[e_3,e_1]_0=0$, $[e_2,e_3]_0=e_3$, $[e_2,e_1]_0=e_1$.
Therefore, this algebra is isomorphic to the initial algebra $\mathfrak g$.
An isomorphism is established by a permutation of the basis elements.
This means that the contraction is improper.

Suppose that $a_{31}=0$ and $a_{21}\neq 0$. The solution of the system~\eqref{EqSystemOnEntriesOfA} gives
$a_{32}=\frac 12 a_{21}$, $a_{21}=-a_{11}$, $a_{22}=-a_{11}$ and $(a_{21}-a_{12})a_{11}=0$.
The constants~$a_{11}$, $a_{22}$, $a_{12}$ and $a_{21}$ can be set to 0, 0, 0, and $-2$,
respectively, by changing the basis with an appropriate matrix~$S$.
As a result, we obtain the canonical commutation relations of the algebra ${\rm sl}(2,\mathbb R)$,
$[e_1,e_2]=e_1$, $[e_2,e_3]=e_3$, $[e_1,e_3]=2e_2$.
The contracted algebra~$\mathfrak g_0$ is isomorphic to the algebra $\mathfrak g_{3.3}$,
which can be seen from its commutation relations, $[e_1,e_2]_0=0$, $[e_2,e_3]_0=e_2$, $[e_1,e_3]_0=e_1$.

In the case $a_{31}=a_{21}=0$ and $a_{32}\neq 0$ the system~\eqref{EqSystemOnEntriesOfA} is reduced to the single equation $a_{32}(a_{11}-a_{22})+a_{12}a_{22}=0$.
Carrying out an admitted basis transformation,
we select certain values of the parameters~$\alpha$ and $\gamma$ of the transformation matrix $S$
in order to set $a_{22}=0$ and~$a_{32}=-1$.
Then the above equation implies that $a_{11}=0$.
Finally, the commutation relations of $\mathfrak g$ take the form $[e_1,e_3]=0$, $[e_2,e_3]=e_3-a_{12}e_1$, $[e_2,e_1]=e_1$,
i.e., $\mathfrak g$ is an almost Abelian algebra associated with the matrix
\[
\left(
\begin{array}{cc}
1&-a_{12}\\
0&1
\end{array}
\right).
\]
The contracted algebra has the same commutation relations as in the previous case, $\mathfrak g_0\sim\mathfrak g_{3.3}$.
If $a_{12}=0$, then the contraction $\mathfrak g\to\mathfrak g_0$ is improper since $\mathfrak g\sim\mathfrak g_{3.3}$.
For $a_{12}\neq 0$, the contraction is equivalent to the unit fall%
\footnote{%
 In the case of $2\times 2$ Jordan blocks, the only possible unit fall is the replacement of the value $1$ in the $(1,2)$th entry by $0$.
 }
 the matrix associated with the algebra~$\mathfrak g\sim\mathfrak g_{3.2}$,
and the resulting matrix defines the algebra $\mathfrak g_{0}\sim\mathfrak g_{3.3}$.

The last case is given by $a_{31}=a_{21}=a_{32}=0$. The single equation remaining in the system~\eqref{EqSystemOnEntriesOfA} is $a_{12}a_{22}=0$.
The commutation relations of the initial and the contracted algebras are respectively
\[\arraycolsep=0ex
\begin{array}{ll}
[e_3,e_1]=a_{11}e_1,             \qquad &[e_3,e_1]_0=0,\\[1ex]
[e_3,e_2]=a_{12}e_1+a_{22}e_2,   \qquad &[e_3,e_2]_0=(a_{11}-a_{22})e_1,\\[1ex]
[e_1,e_2]=0,                     \qquad &[e_1,e_2]_0=0.
\end{array}
\]
Consider subcases depending on values of the remaining parameters.
If $a_{11}\neq a_{22}$, then by selecting a proper value of~$\alpha$ in the transformation matrix~$S$ we can set $a_{12}=0$.
The parameter~$\beta$ is not essential here, and we can choose the zero value for it.
The parameter $\gamma$ can be used for scaling a nonzero linear combination of $a_{11}$ and $a_{22}$ (e.g., $a_{11}-a_{22}$) to the unity.
As a result, we have the contraction of the almost Abelian algebra $\mathfrak g=\mathfrak g_{3.4}$ associated with the diagonal
(but not proportional to the identity matrix) matrix to the three-dimensional Heisenberg algebra~${\rm h}(1)=\mathfrak g_{3.1}$.
If $a_{11}=a_{22}$, the contracted algebra is Abelian, i.e., we have the trivial contraction of an almost Abelian Lie algebra
(one of $\mathfrak g_{3.1}$, $\mathfrak g_{3.2}$, $\mathfrak g_{3.3}$ and $3\mathfrak g_1$, depending on values of $a_{11}=a_{22}$ and~$a_{12}$).

\begin{proposition}
Saletan contractions with the signature $(0;3)$ realize only the following contractions between three-dimensional Lie algebras:
the proper contractions ${\rm sl}(2,\mathbb R)\to\mathfrak g_{3.3}$, $\mathfrak g_{3.2}\to\mathfrak g_{3.3}$, $\mathfrak g_{3.4}\to\mathfrak g_{3.1}$,
and $\mathfrak g_{2.1}\oplus\mathfrak g_1\to\mathfrak g_{3.3}$,
the trivial contractions of $\mathfrak g_{3.1}$, $\mathfrak g_{3.2}$ and $\mathfrak g_{3.3}$ to $3\mathfrak g_1$,
as well as the improper contractions $\mathfrak g_{3.3}\to\mathfrak g_{3.3}$ and $3\mathfrak g_1\to 3\mathfrak g_1$.
\end{proposition}

\section{Conclusion}
The main result of the paper is Theorem~\ref{TheoremOnCanonicalForms}, which describes the canonical form of Saletan contractions.
The proved existence of the canonical form for each Saletan contraction gives a specific finite tuple of non-negative integers
which corresponds to this contraction and is called its signature.
The signature of a Saletan contraction completely defines its canonical form.
Introducing the notion of signature leads to posing several interesting problems, which are related to Saletan contractions.

Thus, for each Lie algebra the set of its possible Saletan contractions is partitioned into the subsets corresponding to different Saletan signatures.
This allows us to pose the problem on describing Saletan contractions with a fixed signature.
The well-known In\"on\"u--Wigner contractions constitute a subclass of Saletan contractions,
which is singled out by Saletan signatures of the form $(n_0; 1,\dots,1)$.
Therefore, the study of Saletan contractions includes, as its simplest part, the study of In\"on\"u--Wigner contractions.
In\"on\"u--Wigner contractions of three- and four-dimensional Lie algebras were exhaustively classified in~\cite{Conatser1972} and~\cite{Huddleston1978}, respectively.
The Saletan contractions with other signatures do not have a connection with algebraic structure of initial and contracted algebras
as direct as In\"on\"u--Wigner contractions do.
This is why the description of general Saletan contractions is a much more difficult problem.

Given a Lie algebra, another problem is finding the tuples of non-negative integers that can be signatures of Saletan contractions of this algebra.
As shown in Remark~\ref{RemarkOnConnectionBtwnSignatureAndAlgStructure},
powers of the value of Saletan matrix at limit value of the contraction parameter form a nested chain of subalgebras of the initial algebra
and signature components are expressed in the terms of the dimensions of these subalgebras.
This claim relates the signatures to the subalgebraic structure of the initial algebra.
At the same time, the presence of a nested chain of subalgebras does not imply the existence of the Saletan contraction associated with this chain.
Additional constraints that admit no clear algebraic interpretation should be taken into account.
Furthermore, even provided that a corresponding contraction exists, there is no known procedure to construct this contraction from the chain of subalgebras.
This significantly differs from In\"on\"u--Wigner contractions since there exists an algorithm
to construct a well-defined In\"on\"u--Wigner contraction starting from any subalgebra of the initial algebra.
The study of Saletan signatures resembles the study of signatures of generalized In\"on\"u--Wigner contractions
\cite{Doebner&Melsheimer1967,Hegerfeldt1967,Lyhmus1969,Popovych&Popovych2009,Popovych&Popovych2010}.
Recall that the signature components of a generalized In\"on\"u--Wigner contraction are diagonal entries
of a diagonal differentiation of the algebra to be contracted, but the converse is not~true.

The notion of Saletan signature may serve as a basis for an algorithm of exhaustive classification of Saletan contractions,
at least in the case of lowest dimensions.
It is known~\cite{Campoamor-Stursberg2008,Nesterenko&Popovych2006,Popovych&Popovych2010,Weimar-Woods1991} that all contractions
between three-dimensional complex (resp.\ real) Lie algebras (except the only contraction ${\rm so}(3)\to{\rm h}(1)$ in the real case)
are realized by usual In\"on\"u--Wigner contractions.
The contraction ${\rm so}(3)\to{\rm h}(1)$ is realized as a generalized In\"on\"u--Wigner contraction, but not as a Saletan one.
In dimension four, the number of contractions that cannot be realized as usual In\"on\"u--Wigner contractions increases crucially.
Moreover, there is one (resp.\ two) contraction between four-dimensional complex (resp.\ real) Lie algebras
that cannot be realized as generalized In\"on\"u--Wigner contractions.
Thus, the question whether these contractions can be realized as Saletan contractions
is the most interesting problem on Saletan contractions of four-dimensional Lie algebras.

\bigskip\par\noindent{\bf Acknowledgements.}
The author is grateful for the hospitality and financial support provided by the University of Cyprus.
The author thanks Roman Popovych for productive and helpful discussions.
The research was supported by the Austrian Science Fund (FWF), project P25064.

\vspace*{-2ex}


\frenchspacing

\begin{thebibliography}{99}\itemsep=-0.15ex
\footnotesize


\bibitem{Burde1999}
Burde D.,
Degenerations of nilpotent Lie algebras,
{\it J. Lie Theory} {\bf 9} (1999), 193--202.

\bibitem{Burde2005}
Burde D.,
Degenerations of 7-dimensional nilpotent Lie algebras,
{\it Comm. Algebra} {\bf 33} (2005), 1259--1277; arXiv:math.RA/0409275.

\bibitem{Burde&Steinhoff1999}
Burde D. and Steinhoff C.,
Classification of orbit closures of 4-dimensional complex Lie algebras,
{\it J.~Algebra} {\bf 214} (1999), 729--739.

\bibitem{Campoamor-Stursberg2008}
Campoamor-Stursberg R.,
Some comments on contractions of Lie algebras,
{\it Adv. Studies Theor. Phys.} {\bf 2} (2008), 865--870.

\bibitem{Carinena&Grabowski&Marmo2000}
Cari\~nena J., Grabowski J. and Marmo G.,
Contractions: Nijenhuis and Saletan tensors for general algebraic structures,
{\it J. Phys. A: Math. Gen.} {\bf 34} (2001), 3769--3789.

\bibitem{Conatser1972}
Conatser C.W.,
Contractions of the low-dimensional Lie algebras,
{\it J. Math. Phys.} {\bf 13} (1972), 196--203.

\bibitem{Doebner&Melsheimer1967}
Doebner H.D. and Melsheimer O.,
On a class of generalized group contractions,
{\it Nuovo Cimento A(10)} {\bf 49} (1967), 306--311.

\bibitem{Grunewald&Halloran1988}
Grunewald F. and O'Halloran J.,
Varieties of nilpotent Lie algebras of dimension less than six,
{\it J. Algebra} {\bf 112} (1988), 315--325.

\bibitem{Hegerfeldt1967}
Hegerfeldt G.C.,
Some properties of a class of generalized In\"on\"u--Wigner contractions,
{\it Nuovo Cimento A(10)} {\bf 51} (1967), 439--447.

\bibitem{Huddleston1978}
Huddleston P.L.,
In\"on\"u--Wigner contractions of the real four-dimensional Lie algebras,
{\it J. Math. Phys.} {\bf 19} (1978), 1645--1649.

\bibitem{Inonu&Wigner1953}
In\"on\"u E. and Wigner E.P.,
On the contraction of groups and their representations,
{\it Proc. Nat. Acad. Sci. U.S.A.} {\bf 39} (1953), 510--524.

\bibitem{Inonu&Wigner1954}
In\"on\"u E. and Wigner E.P.,
On a particular type of convergence to a singular matrix,
{\it Proc. Nat. Acad. Sci. U.S.A.} {\bf 40} (1954), 119--121.

\bibitem{Kupczynski1969}
Kupczy\`nski M.,
On the generalized Saletan contractions,
{\it Commun. Math. Phys.} {\bf 13} (1969), 154--162.

\bibitem{Lauret2003}
Lauret J.,
Degenerations of Lie algebras and geometry of Lie groups,
{\it Differential Geom. Appl.} {\bf 18} (2003), 177--194.

\bibitem{Levy-Nahas1967}
Levy-Nahas M.,
Deformation and contraction of Lie algebras,
{\it J. Math. Phys.} {\bf 8} (1967), 1211--1222.

\bibitem{Lyhmus1969}
L\~ohmus Ja.H.,
Limit (contracted) Lie groups,
{\it Proceedings of the Second Summer School on the Problems of the Theory of Elementary Particles
(Otep\"a\"a, 1967),  Part IV, Inst. Fiz. i Astronom. Akad. Nauk Eston. SSR, Tartu}, 1969, 3--132 (Russian).

\bibitem{Mimura&Ikushima1977}
Mimura F. and Ikushima A.,
Structure of contracted Lie algebras,
{\it Bull. Kyushu Inst. Tech. (M. \& N. S.)} (1978), no. 25, 1--7.

\bibitem{Mubarakzyanov1963a}
Mubarakzyanov G.M.,
On solvable Lie algebras,
{\it Izv. Vys. Ucheb. Zaved. Matematika} (1963), no. 1 (32), 114--123 (Russian).

\bibitem{Nesterenko&Popovych2006}
Nesterenko M. and Popovych R.O.,
Contractions of low-dimensional Lie algebras,
{\it J. Math. Phys.} {\bf 47} (2006), 123515, 45 pp.; arXiv:math-ph/0608018.

\bibitem{Patera&Winternitz1977}
Patera J. and Winternitz P.,
Subalgebras of real three and four-dimensional Lie algebras,
{\it J. Math. Phys.} {\bf 18} (1977), 1449--1455.

\bibitem{Popovych&Popovych2009}
Popovych D.R. and Popovych R.O.,
Equivalence of diagonal contractions to generalized IW-contractions with integer exponents,
{\it Linear Algebra Appl.} {\bf 431} (2009), 1096--1104; arXiv:0812.4667.

\bibitem{Popovych&Popovych2010}
Popovych D.R. and Popovych R.O.,
Lowest dimensional example on non-universality of generalized In\"on\"u--Wigner contractions,
{\it J. Algebra} {\bf 324} (2010), 2742--2756, arXiv:0812.1705.

\bibitem{Popovych&Boyko&Nesterenko&Lutfullin2005}
Popovych R.O., Boyko V.M., Nesterenko M.O. and Lutfullin M.W.,
Realizations of real low-dimensional Lie algebras,
2005, arXiv:math-ph/0301029v7, 39 pp.
(extended and revised version of paper {\it J. Phys. A: Math. Gen.} {\bf 36} (2003), 7337--7360).

\bibitem{Saletan1961}
Saletan E.J.,
Contraction of Lie groups,
{\it J. Math. Phys.} {\bf 2} (1961), 1--21.

\bibitem{Segal1951}
Segal I.E.,
A class of operator algebras which are determined by groups,
{\it Duke Math. J.} {\bf 18} (1951), 221--265.

\bibitem{Weimar-Woods1991}
Weimar-Woods E.,
The three-dimensional real Lie algebras and their contractions,
{\it J. Math. Phys.} {\bf 32} (1991), 2028--2033.

\end{thebibliography}
\end{document}